\documentclass[12pt]{amsart}
\usepackage[all]{xy}
\usepackage{amssymb}
\usepackage{amsthm}
\usepackage{hyperref}
\usepackage{todonotes}
\hypersetup{colorlinks=true,linkcolor=blue,citecolor=magenta}
\usepackage{amsmath}
\usepackage{mathtools}
\usepackage{amscd, enumerate}
\usepackage{verbatim}
\usepackage{eurosym}
\usepackage{float}
\usepackage{color}
\usepackage{physics}
\usepackage{cases}
\usepackage{bm}
\usepackage{dcolumn}
\usepackage{makecell}
\usepackage[mathscr]{eucal}
\usepackage[all]{xy}
\usepackage{hyperref}
\usepackage{bbm}
\usepackage{comment}
\usepackage{adjustbox}
\usepackage{caption, tabu, multicol}
\usepackage[textheight=8.77in, textwidth=6.8in]{geometry}

\usepackage[]{breqn}

\newtheorem*{thm*}{Theorem}
\newtheorem*{conj*}{Conjecture}

\newtheorem*{remark}{Remark}

\newtheorem{theorem}{Theorem}[section]

\newtheorem{lemma}[theorem]{Lemma}

\newtheorem{proposition}[theorem]{Proposition}

\newtheorem{corollary}[theorem]{Corollary}

\newcommand{\Z}{\mathbb{Z}}

\newcommand{\SL}{\operatorname{SL}}

\newcommand{\ep}{\varepsilon}

\makeatletter
\DeclarePairedDelimiterX{\pmodx}[1]{(}{)}{{\operator@font mod}\mkern6mu#1}
\renewcommand{\pmod}{%
  \allowbreak
  \if@display\mkern18mu\else\mkern8mu\fi
  \pmodx
}
\makeatother

\numberwithin{equation}{section}

\begin{document}
\title{Some Remarks on Small Values of $\tau(n)$}

\author[K. Lakein, A. Larsen]{Kaya Lakein and Anne Larsen}
\address{Department of Mathematics, Stanford University, Stanford, CA 94305}
\email[K. Lakein]{epi2@stanford.edu}
\address{Department of Mathematics, Harvard University, Cambridge, MA 02138}
\email[A. Larsen]{larsen@college.harvard.edu}

\begin{abstract}
A natural variant of Lehmer's conjecture that the Ramanujan $\tau$-function never vanishes asks whether, for any given integer $\alpha$, there exist any $n \in \Z^+$ such that $\tau(n) = \alpha$. A series of recent papers excludes many integers as possible values of the $\tau$-function using the theory of primitive divisors of Lucas numbers, computations of integer points on curves, and congruences for $\tau(n)$. We synthesize these results and methods to prove that if $0 < \abs{\alpha} < 100$ and $\alpha \notin T := \{2^k, -24,-48, -70,-90, 92, -96\}$, then $\tau(n) \neq \alpha$ for all $n > 1$.
Moreover, if $\alpha \in T$ and $\tau(n) = \alpha$, then $n$ is square-free with prescribed prime factorization. Finally, we show that a strong form of the Atkin-Serre conjecture implies that $\abs{\tau(n)} > 100$ for all $n > 2$.
\end{abstract}

\maketitle

\section{Introduction and Statement of Results}\label{sec: intro}
In his 1916 paper titled ``On certain arithmetical functions,'' Ramanujan defined the function $\tau(n)$ to be the Fourier coefficients of the normalized weight 12 cusp form for $\SL_2(\Z)$ given by
\begin{equation}
    \Delta(z) = \sum_{n=1}^\infty \tau(n)q^n:= q\prod_{n=1}^\infty (1-q^n)^{24} = q-24q^2+252q^3-1472q^4+4830q^5-\dots,
\end{equation}
where $q := e^{2\pi i z}$. The $\tau$-function plays an important role in the theory of modular forms, yet some of its basic properties remain unknown. For instance, Lehmer's 1947 conjecture \cite{Lehmer} that $\tau(n)$ never vanishes remains open. 

A natural variant of Lehmer's conjecture asks whether, for any given integer $\alpha$, there exist any $n \in \Z^+$ such that $\tau(n) = \alpha$. Much is known for odd $\alpha$ due to the fact that 
\begin{equation}\label{eq: deltamod2}
    \Delta(z) \equiv \sum_{n=0}^\infty q^{(2n+1)^2} \pmod 2.
\end{equation}

In 1987, Murty, Murty, and Shorey \cite{MMS} showed that for any odd integer $\alpha$, we have $\tau(n) = \alpha$ for at most finitely many $n \in \Z^+$. It was not until 2013, however, that the first integers were ruled out as values of $\tau(n)$, when Lygeros and Rozier \cite{LR} proved that $\tau(n) \neq \pm 1$ for all $n > 1$. In 2020, Balakrishnan, Craig, Ono, and Tsai \cite{BCO,BCOT} introduced new methods to the problem, and showed that
$$\tau(n) \notin \{\pm 1, \pm 3, \pm 5, \pm 7, \pm 13, \pm 17, -19, \pm 23, \pm 37, \pm 691\}.$$
Later that year, Hanada and Madhukara \cite{HM} additionally proved that $$\tau(n) \notin \{-9,\pm 15,\pm 21,-25,-27,-33,\pm 35, \pm 45, \pm 49,-55,\pm 63, \pm 77, -81,\pm 91\},$$ and Dembner and Jain \cite{DJ} showed that $\tau(n) \neq \pm \ell$, where $\ell < 100$ is any odd prime. Shortly thereafter, Bennett, Gherga, Patel, and Siksek \cite{Bennett} proved that the same is true for any positive power of $\ell$, and that $\tau(n) \neq \pm 3^a 5^b 7^c 11^d$ for any $a,b,c,d \in \Z_{\ge 0}$ and any $n > 1$. Using some of these results concerning odd primes, Balakrishnan, Ono, and Tsai \cite{BOT} were able to make progress eliminating non-zero even integers as values of the $\tau$-function. For odd primes $\ell$, they showed that $$\tau(n) \notin \{\pm 2\ell:\, \ell < 100\} \cup \{\pm 2\ell^2:\, \ell < 100\} \cup \{\pm 2\ell^3:\, \ell < 100 \text{ with } \ell \neq 59\}.$$

The methods used to obtain these results include the theory of primitive divisors of Lucas numbers, the paucity of integer points on hyperelliptic curves and Thue equations, and the congruences for the $\tau$-function (see \eqref{eq: congs} and \eqref{eq: pmod}) arising from the theory of 2-dimensional Galois representations due to Serre and Swinnerton-Dyer \cite{Mods}. The objective of this paper is to synthesize these results and exclude as many additional non-zero $\tau$-values with absolute value less than $100$ as possible given current methods, as well as to explain the limitations of these methods.
In particular, we prove the following:
\begin{theorem}\label{thm: main}
The following are true:
\begin{enumerate}[(i)]
    \item If $\tau(n) \in \{2,-24, -70,-90,92\}$, then $n = p$ is prime.
    \item If $\tau(n) = 2^k < 100$, then $n = p_1\cdots p_k$ for distinct primes $p_i$ with $\tau(p_i) = 2$ for all $i$.
    \item If $\tau(n) =-48$, then $n = p_1p_2$ for primes $p_1,p_2$ with $\tau(p_1)=2$ and $\tau(p_2)=-24$. 
    \item If $\tau(n) = -96$, then $n = p_1p_2p_3$ for distinct primes with $\tau(p_1) = \tau(p_2) = 2$ and $\tau(p_3) = -24$.
    \item For all $\alpha \in \Z$ with $0 < \abs{\alpha} < 100$ not listed in (i)-(iv), we have that $\tau(n) \neq \alpha$ for all $n > 1$.
\end{enumerate}
\end{theorem}

\begin{remark}
Assuming the Lang-Trotter conjecture~\cite{LangTrotter}, there are only finitely many primes $p$ with $\tau(p) = \alpha$ for any integer $\alpha$, in which case Theorem~\ref{thm: main} shows that there are only finitely many $n$ with $0<\abs{\tau(n)}<100$.
\end{remark}

The obstacle to proving that $\abs{\tau(n)} \ge 100$ for all $n>2$ is that the congruences given in \cite{Mods}, which are currently the only known means for ruling out possible values of $\tau(p)$, do not eliminate the exceptional values in Theorem~\ref{thm: main} (i) (whose products give rise to the exceptional values in (ii)-(iv)). New methods will therefore be needed to better understand possible values of $\tau(p)$. Recent progress on the Atkin-Serre conjecture provides hope that $\abs{\tau(p)}$ could be bounded below for all primes. Using an effective form of the Sato-Tate conjecture, Gafni, Thorner, and Wong \cite{Jesse} recently showed that 
\begin{equation}\label{eq: StrongJesse}
2p^{\frac{11}{2}}\,\frac{\log \log p}{\sqrt{\log p}} < \abs{\tau(p)}
\end{equation}
for 100 percent of (but not all) primes $p$. We call the set $S$ of all primes contradicting \eqref{eq: StrongJesse} the \textit{exceptional set}. If $p \notin S \cup \{2,3\}$, then \eqref{eq: StrongJesse} shows that $\abs{\tau(p)}>100$, and hence we immediately obtain the following corollary of Theorem~\ref{thm: main}:

\begin{corollary}\label{thm: dumbAT}
If $n>1$ and $0 < \abs{\tau(n)} < 100$, then every odd prime factor of $n$ is in $S$.
\end{corollary}
Of course, to prove that $\abs{\tau(n)}\ge 100$ for $n > 2$, we would need a stronger statement, such as an effective version of the Atkin-Serre conjecture for $\tau(n)$ \cite{Atkin-Serre}.

\begin{conj*}[Atkin-Serre $+\,\ep$] For any $\ep> 0$, there are constants $c(\ep), d(\ep) > 0$ such that for all primes $p > d(\ep)$, we have
\begin{equation}\label{eq: AtkinSerre}
\abs{\tau(p)} \ge\, c(\ep)\cdot p^{\frac{9}{2}-\ep}.
\end{equation}
\end{conj*}


If one could find an $\ep > 0$ and corresponding constants $c(\ep)$ and $d(\ep)$ satisfying this form of the Atkin-Serre conjecture, then proving that $\abs{\tau(n)} < 100$ if and only if $n \le 2$ would reduce to a finite calculation. Since such a statement is currently out of reach, we instead carry out numerical experiments that suggest an $\ep > 0$ and constants $c(\ep),d(\ep)$ which could conceivably satisfy this condition. In particular, we let $P_n$ denote the set of the first $n$ primes, and $p_{\mathrm{min}}(\ep,n) \in P_n$ denote the prime at which the expression $\abs{\tau(p)}p^{-\frac 92+\ep}$ is minimized.
The values of $p_{\mathrm{min}}(\ep,n)$ and $\min_{p \in P_n} \abs{\tau(p)}p^{-\frac 92+\ep}$ for $n = 30000$ and different small values of $\ep$ are given in the following table:

\begin{table}[H]
    \centering
    \begin{tabular}{|c|c|c|c|c|c|}
    \hline
    $\ep$ & $0.001$ & $0.01$ & $0.1$ & $0.25$ & $0.5$ \\ \hline
    $p_{\mathrm{min}}(\ep, 30000)$ & $43$ & $43$ & $43$ & $2$ & $2$ \\ \hline
    $\min_{p \in P_{30000}}\abs{\tau(p)}p^{-\frac 92+\ep}$ & $0.766 \dots$ & $0.793 \dots$ & $1.112 \dots$ & $1.261 \dots $ & $1.5$\\ \hline
    \end{tabular}
\label{tbl: epexperiments}
\vspace{2mm}
\caption{Numerics for the Atkin-Serre conjecture $+\, \ep$ with the first $30000$ primes}
\end{table}

\noindent Assuming, for instance, that $\abs{\tau(p)} \ge 0.793p^{\frac 92-0.01}$ for all primes $p$ (so that $\ep = 0.01$ satisfies the Atkin-Serre conjecture with $c(\ep) = 0.793$, $d(\ep) = 0$), we have $\abs{\tau(n)} < 100$ if and only if $n \le 2$.

\begin{remark}
It would be interesting to carry out further numerics with the goal of finding explicit constants $c(\ep)$ and $d(\ep)$ for which some form of the Atkin-Serre conjecture for arbitrary newforms of integer weight $k \ge 4$ without complex multiplication appears to hold.
\end{remark}

The proof of Theorem~\ref{thm: main} follows the approach of \cite{BCOT} and \cite{BOT} for odd and even values, respectively. For each $\alpha$, we use the fact that smaller values have been excluded as values of the $\tau$-function to show that $\tau(n) = \alpha$ implies $n=p^a$ for some prime $p$. Due to the classification of primitive divisors of Lucas sequences in \cite{BHV, Abouzaid}, we are able to restrict the list of possibilities for $a$, each of which we then eliminate using the congruences for the $\tau$-function listed in \cite{Mods}. (In contrast to previous work, where $\tau(p^a) = \alpha$ for $a>1$ is excluded by solving Diophantine equations, we found these congruences to be sufficient in all cases of interest.)

This paper is organized as follows: In Section~\ref{sec: prelims}, we recall some facts about Lucas numbers and their primitive prime divisors, and connect these to values of the $\tau$-function at prime powers. In Sections~\ref{sec: odd} and \ref{sec: even}, we exclude the odd and even integers indicated in Theorem~\ref{thm: main} as possible values of $\tau(n)$, respectively. Finally, we complete the proof of Theorem~\ref{thm: main} in Section~\ref{sec: proofofmain}. 

\section*{Acknowledgements}
We would like to thank Ken Ono for suggesting and advising this project, and for many valuable comments. We also thank William Craig, Badri Pandey, Wei-Lun Tsai, and the referee for their helpful suggestions. Finally, we are grateful for the generous support of the National Science Foundation (DMS 2002265 and DMS 205118), the National Security Agency (H98230-21-1-0059), the Thomas Jefferson Fund at the University of Virginia, and the Templeton World Charity Foundation. This research was conducted as part of the Number Theory Research Experience for Undergraduates at the University of Virginia.

\section{Preliminaries}\label{sec: prelims}
In this section, we recall some facts about primitive prime divisors of Lucas sequences, as well as several properties of the $\tau$-function, which we will need to prove Theorem~\ref{thm: main}.

\subsection{Lucas sequences}
If $\alpha$ and $\beta$ are algebraic integers such that $\alpha+\beta$ and $\alpha\beta$ are relatively prime non-zero integers and $\frac{\alpha}{\beta}$ is not a root of unity, then we call $(\alpha,\beta)$ a \textit{Lucas pair}, and we define the \textit{Lucas sequence} $u_n(\alpha,\beta) = \{u_1 = 1,\, u_2 = \alpha+\beta, \dots\}$ by 
\begin{equation*}
u_n(\alpha,\beta) := \frac{\alpha^n-\beta^n}{\alpha-\beta}.
\end{equation*}
Lucas sequences have the following divisibility property:
\begin{proposition}[Proposition 2.1 (ii) of \cite{BHV}]\label{prop: relativedivisibility}
    If $d \mid n$, then $u_d(\alpha,\beta) \mid u_n(\alpha,\beta)$.
\end{proposition}

For any prime $\ell$, let $m_\ell(\alpha,\beta)$ denote the smallest $n \ge 2$ such that $\ell \mid u_n(\alpha,\beta)$. Then we have $m_\ell(\alpha,\beta) = 2$ if and only if $\alpha +\beta \equiv 0 \pmod\ell$, and $m_\ell(\alpha,\beta)$ fulfills the following condition:

\begin{proposition}[Proposition 2.3 of \cite{BOT}]\label{prop: possiblea}
    If $\ell$ is an odd prime with $2 < m_\ell(\alpha,\beta) < \infty$, then the following are true:
    \begin{enumerate}[(i)]
        \item If $\ell \mid (\alpha-\beta)^2$, then $m_\ell(\alpha,\beta) = \ell$.
        \item If $\ell \nmid (\alpha-\beta)^2$, then $m_\ell(\alpha,\beta)\mid \ell-1$ or $m_\ell(\alpha,\beta) \mid \ell+1$.
    \end{enumerate}
\end{proposition}

If $(\alpha,\beta)$ is a Lucas pair, then a prime number $p$ is called a \textit{primitive prime divisor} of $u_n(\alpha,\beta)$ if $p \mid u_n(\alpha,\beta)$, but $p \nmid (\alpha-\beta)^2u_1(\alpha,\beta)\cdots u_{n-1}(\alpha,\beta)$.  If $n > 2$ and $u_n(\alpha,\beta)$ does not have a primitive prime divisor, then $u_n(\alpha,\beta)$ is called \textit{defective}. Bilu, Hanrot, and Voutier \cite{BHV} showed that every Lucas number $u_n(\alpha,\beta)$ with $n > 30$ has a primitive prime divisor. Their work, together with a subsequent paper of Abouzaid \cite{Abouzaid}, completely classifies defective Lucas numbers. 

\subsection{Properties of $\texorpdfstring{\tau(n)}{tau(n)}$}
Important properties of the $\tau$-function include the Hecke multiplicativity established by Mordell \cite{Mordell}, and a deep theorem of Deligne \cite{Deligne1, Deligne2} that gives an upper bound for $\abs{\tau(p)}$.
\begin{theorem}\label{thm: properties}
The following are true:
\begin{enumerate}[(i)]
    \item If $\gcd(m,n) = 1$, then $\tau(mn) = \tau(m)\tau(n)$.
    \item If $p$ is prime and $a \ge 2$, then 
    \begin{equation*}
    \tau(p^a) = \tau(p)\tau(p^{a-1})-p^{11}\tau(p^{a-2}).
    \end{equation*}
    \item If $p$ is prime and $\alpha_p, \beta_p$ are roots of $F_p(X) := X^2-\tau(p)X+p^{11},$ then 
    $$\tau(p^a) = u_{a+1}(\alpha_p,\beta_p) = \frac{\alpha_p^{a+1}-\beta_p^{a+1}}{\alpha_p-\beta_p}.$$
    Moreover, $\abs{\tau(p)} \le 2p^{\frac{11}{2}}$, and $\alpha_p$ and $\beta_p$ are complex conjugates.
\end{enumerate}
\end{theorem}

\noindent In particular, Theorem~\ref{thm: properties} (i) has the following corollary, which we use to prove Theorem~\ref{thm: main}:

\begin{corollary}\label{lem: npaodd}
    Suppose $\alpha \in \Z$ is such that for any factorization $\alpha = \beta \gamma$, we have either $\tau(n) \ne \beta$ or $\tau(n) \ne \gamma$ for $n>1$. If $\tau(n) = \alpha$, then $n=p^a$ for some prime $p$.
\end{corollary}

If $p \mid \tau(p)$, then $p^a \mid \tau(p^a)$ for all $a \in \Z^+$ by the recursion in Theorem~\ref{thm: properties} (ii). Since $\abs{\tau(n)} > 100$ for $3 \le n \le 100$, we have that $\alpha \neq \tau(p^a)$ for any $p$ such that $p \mid \tau(p)$ when $0 < \abs{\alpha} < 100$ (except for $\tau(2) = -24$).
Therefore, for the purposes of this paper, it suffices to assume that $p \nmid \tau(p)$, in which case $\{1,\tau(p), \tau(p^2),\dots\}$ forms a Lucas sequence by Theorem~\ref{thm: properties} (iii).

\begin{lemma}\label{lemma: defectivity}
   If $p \nmid \tau(p)$ and $\abs{\tau(p^a)} < 100$, then $\tau(p^a)$ is not a defective term in the Lucas sequence $\{1,\tau(p),\tau(p^2),\dots\}$.
\end{lemma}

\begin{proof}
    A complete classification of the defective Lucas numbers is given in Tables 1 (sporadic examples) and 2 (parametrized families) of \cite{BCOT}. For the particular Lucas sequence $u_n = \tau(p^{n-1})$, the variables $(A,B)$ given in these tables correspond to $(\tau(p),p^{11})$. Since none of the values of $B$ listed in Table 1 are $11$th powers, the sporadic examples listed here do not apply to the Lucas sequence $u_n$.
    Now, by the results of Lygeros and Rozier \cite{LR} as well as Bennett, Gherga, Patel, and Siksek \cite{Bennett} stated in Section~\ref{sec: intro}, we have that $\tau(p^a) \neq -1, \pm 3^r$ as in rows 1 and 2 of Table 2. Rows 3, 5, and 7 give rise to a defective term only if $m = \tau(p)$ is odd; however, by \eqref{eq: deltamod2} we know that $\tau(p)$ is even for all primes. If row 4 of Table 2 gives rise to a defective term $u_4 = \tau(p^3)$, then $(p,\pm m) = (p,\pm \tau(p)) \in B_{3,k}^\pm$ (as defined on p.~20 of \cite{BCOT}),
    meaning that $\tau(p)^2 = 2p^{11}\pm 2$. Using Theorem~\ref{thm: properties} (ii), we see that then
    \begin{align*}
        \abs{\tau(p^3)} = \abs{\tau(p)(\tau(p)^2-2p^{11})} = 2 \abs{\tau(p)} = 2 (2p^{11} \pm 2)^{\frac{1}{2}} > 100.
    \end{align*}
    Finally, row 6 yields a defective term $u_6$ only if $\abs{\tau(p)} \geq 6$. However, if this is the case, then
    $$\abs{\tau(p^5)} = \abs{u_6} = \abs{\tau(p)} (2\tau(p)^2\pm 3) > 100.$$
\end{proof}

\section{Odd Values of \texorpdfstring{$\tau(n)$}{tau(n)}}\label{sec: odd}
The goal of this section is to prove the following theorem concerning odd inadmissible values for the $\tau$-function:
\begin{theorem}\label{thm: mainodd}
If $\alpha$ is any odd integer with $\abs{\alpha} < 100$, then $\tau(n) \neq \alpha$ for any $n > 1$.
\end{theorem}

\begin{proof}
Using the previous results described in Section~\ref{sec: intro}, it suffices to show that
\begin{equation}\label{eq: oddvals}
    \tau(n) \notin \{\pm 39, \pm 51, \pm 57, \pm 65, \pm 69, \pm 85, \pm 87, \pm 93, \pm 95\}
\end{equation}
for $n \in \Z^+$. 
We note that all the values $\alpha$ in \eqref{eq: oddvals} satisfy the conditions of Corollary~\ref{lem: npaodd} by the results stated in Section~\ref{sec: intro}.
Thus if $\tau(n) = \alpha$, we must have $n = p^a$ for some prime $p$ and $a \ge 1$, where $p$ is odd and $a$ is even by \eqref{eq: deltamod2}. As explained in Section~\ref{sec: prelims}, it suffices to consider the case $p \nmid \tau(p)$, so that the values $\{1, \tau(p), \tau(p^2), \cdots\}$ form a Lucas sequence. By Lemma~\ref{lemma: defectivity}, we have that $\alpha = \tau(p^a)$ is a non-defective term in this sequence, so $\alpha$ has some prime divisor $\ell$ such that $m_{\ell} = a+1$. Hence by Proposition~\ref{prop: possiblea}, we have that $(a+1) \mid \ell(\ell^2-1)$. Moreover, we note that $a+1$ must be prime. Indeed, if $a+1$ had a non-trivial divisor $b$, then by Proposition~\ref{prop: relativedivisibility}, we would have $\tau(p^{b-1}) \mid \tau(p^a)$, but by non-defectivity of $\tau(p^a)$ and the odd $\tau$-values previously ruled out, this is impossible.
Considering each possible primitive prime divisor $\ell$ of $\alpha$ individually, we use these conditions on $a+1$ to obtain a finite list of possible values for $a$ (which is given in Table~\ref{table: congruences} below for each $\alpha$ in \eqref{eq: oddvals}).

To rule out each of these values of $a$,
we use the following well-known congruences \cite{Ramanujan, Mods}:
\begin{align}\label{eq: congs}
\begin{split}
    \tau(n) \equiv \begin{dcases} 
    \sigma_{11}(n) \qquad \quad \pmod{691}, \\
    0 \qquad\, \qquad \quad \pmod{23} &\text{ if } n = p \text{ and } p^{11} \equiv -1 \pmod{23}, \\
    -1, 2 \,\,\, \qquad \quad \pmod{23} &\text{ if } n = p \text{ and } p^{11} \equiv 1 \,\,\,\,\, \pmod{23}, \\
    n^{-610} \sigma_{1231}(n) \pmod{3^6} & \text{ if } n \not\equiv 0 \pmod{3}.
    \end{dcases}
\end{split}
\end{align}

If $n = p^a$, then the mod $691$ congruence in \eqref{eq: congs} gives
$$ \tau(p^a) \equiv \sigma_{11}(p^a) = \sum_{i=0}^a p^{11i} \pmod{691}.$$
In particular, if the equation $\sum_{i=0}^a x^i \equiv \alpha \pmod{691}$ has no solutions, then $\alpha \ne \tau(p^a)$ for any prime $p$.
Checking this for each $(\alpha,a)$ in Table~\ref{table: congruences}, we eliminate the pairs indicated in the table.

By the mod 23 congruence, we have $\tau(p) \equiv 0, -1, 2 \pmod{23}$ for all $p \ne 23$, where $\tau(p) \equiv 0 \pmod{23}$ exactly when $p^{11} \equiv -1 \pmod{23}$.
Combining this with the recurrence in Theorem~\ref{thm: properties}~(ii), we can find a list of possible values for $\tau(p^a)$ mod 23, as long as $p \ne 23$. Checking whether $\alpha$ is on this list for each pair in Table~\ref{table: congruences} not already eliminated by the mod 691 congruence, and then checking numerically that $|\tau(23^a)|>10^{14} > \alpha$ for each $a$ in the table, we eliminate the pairs shown in Table~\ref{table: congruences}.

Finally, if for all possible equivalence classes of $3 \nmid p \pmod{3^6}$ we have $p^{-610a} \sigma_{1231}(p^a) \not\equiv \alpha \pmod{3^6}$, then $\alpha \ne \tau(p^a)$ for $p \ne 3$. Furthermore, for all $\alpha$ in the table and any $a \ge 1$, we have $\alpha \not\equiv \tau(3^a) \equiv 0 \pmod{9}$.
This eliminates all remaining pairs $(\alpha, a)$ in Table~\ref{table: congruences}.

\begin{table}[H]
\begin{adjustbox}{width=\columnwidth,center}
    \centering
    \begin{tabular}{|c|l|l|l|}
    \hline
     $\abs{\alpha}$ & $a$ & $\mathrm{Congruences} \,\, \mathrm{for} \,\, \tau(p^a) \ne \abs{\alpha}$ & $\mathrm{Congruences} \,\, \mathrm{for} \,\, \tau(p^a) \ne -\abs{\alpha}$\\ \hline
     $39 = 3\cdot 13$ & $2,6,12$ & $2,6,12 \pmod{691}$ & $2,6 \pmod{691}, 12 \pmod{23}$ \\ \hline
     $51 = 3\cdot 17$ & $2,16$ & $2,16 \pmod{23}$ & $2,16 \pmod{23}$ \\ \hline
     $57 = 3\cdot 19$ & $2,4,18$ & $2,18 \pmod{23}, 4 \pmod{691}$ & $2 \pmod{691}, 4,18 \pmod{23}$ \\ \hline
     $65 = 5\cdot 13$ & $2,4,6,12$ & $2,4,6,12 \pmod{691}$ & $2 \pmod{691}, 4,6,12 \pmod{23}$ \\ \hline
     $69 = 3\cdot 23$ & $2,10,22$ & $2,22 \pmod{3^6}, 10 \pmod{23}$ & $2,22 \pmod{3^6}, 10 \pmod{691}$ \\ \hline
     $85 = 5\cdot 17$ & $2,4,16$ & $2, 4, 16 \pmod{691}$ & $2 \pmod{691}, 4, 16 \pmod{23}$ \\ \hline
     $87 = 3\cdot 29$  & $2,4,6,28$ & $2, 6, 28 \pmod{23}, 4 \pmod{691}$ & $2 \pmod{691}, 4 \pmod{3^6}, 6,28 \pmod{23}$ \\ \hline
     $93 = 3\cdot 31$ & $2,4,30$ & $2 \pmod{691}, 4, 30 \pmod{3^6}$ & $2 \pmod{23}, 4 \pmod{3^6}, 30 \pmod{691}$ \\ \hline
     $95 = 5\cdot 19$  & $2,4,18$ & $2 , 4 , 18 \pmod{691}$ & $2 \pmod{691}, 4, 18 \pmod{23}$ \\ \hline
    \end{tabular}
\end{adjustbox}
\vspace{2mm}
\caption{Congruences showing that $\tau(p^a) \ne \alpha$ for each pair $(\alpha,a)$}\label{table: congruences} 
\end{table}
\vspace{-1cm}
\end{proof}

\section{Even Values of \texorpdfstring{$\tau(n)$}{tau(n)}}\label{sec: even}
The main result of this section is the following theorem concerning even inadmissible values for the $\tau$-function: 
\begin{theorem}\label{thm: maineven}
If $\alpha$ is any even integer with $\abs{\alpha} < 100$ such that $$\alpha \notin\{2^k:\, k \in \Z^+\} \cup \{0,-24,-48,-70,-90,92,-96\},$$ then $\tau(n) \neq \alpha$ for any $n \in \Z^+$.
\end{theorem}

\begin{proof}
Given the previous results stated in Section~\ref{sec: intro}, it suffices to show that
\begin{align}\label{eq: eventaulist}
\begin{split}
\tau(n) \notin \{&-2, -4, -8, \pm 12, -16, \pm 20, 24, \pm 28, \pm 30, -32, \pm 36, \pm 40, \pm 42, \pm 44, 48, \\
&\pm 52,\pm 56, \pm 60, -64, \pm 66, \pm 68, 70, \pm 72, \pm 76, \pm 78, \pm 80, \pm 84, \pm 88, 90, -92, 96\},
\end{split}
\end{align}
for all $n \in \Z^+$. We generalize the proof of Theorem 1.1 in \cite{BOT}, and apply the following argument to the values $\alpha$ in \eqref{eq: eventaulist} inductively starting with the $\alpha$ of smallest absolute value (for example, we need to know that $12$ is not an admissible value for the $\tau$-function in order to show that $\tau(n) \neq 24$ for all $n \in \Z^+$). Suppose for the sake of contradiction that for some $\alpha$ in \eqref{eq: eventaulist}, there exists a positive integer $n$ such that $\tau(n) = \alpha$. 
By the induction hypothesis, the results in Section~\ref{sec: intro}, and Theorem~\ref{thm: mainodd},
Corollary~\ref{lem: npaodd} applies, and so it follows that $n = p^a$ for some prime $p$ and some $a \in \Z^+$. As explained in Section~\ref{sec: prelims}, we can assume that $p \nmid \tau(p)$, so in particular $p \ne 2$.

Since $\tau(p^a) = \alpha \equiv 0 \pmod{2}$, it follows by \eqref{eq: deltamod2} that $a$ must be odd. We claim that this implies $a = 1$. Assume otherwise. By Lemma~\ref{lemma: defectivity}, we know $\alpha = \tau(p^a)$ is not a defective term in the Lucas sequence $\{1,\tau(p),\tau(p^2),\dots\}$. Therefore, $\tau(p^a)$ must have a primitive prime divisor $\ell$, which must furthermore be odd since $2 \mid \tau(p)$.
Suppose now that $a+1$ has a non-trivial divisor $b$. Then $\tau(p^{b-1})$ divides $\tau(p^a) = \alpha$ by Proposition~\ref{prop: relativedivisibility}.

Using Theorem~\ref{thm: mainodd}, the results stated in Section~\ref{sec: intro}, and the induction hypothesis, we check that for any odd prime $\ell$ dividing $\alpha$, any factor of $\alpha$ not divisible by $\ell$ which is not a power of 2 is not an admissible value for the $\tau$-function. Thus, we must have $\tau(p^{b-1}) = 2^k$ for some $1 \le k \le 6$.
By Lemma~\ref{lemma: defectivity} and the fact that $2 \mid \tau(p)$, it follows that $b = 2$. We conclude that $2$ is the only possible non-trivial factor of $a+1$, so either $a+1$ is prime or $a+1 = 4$.

If $a+1$ is prime, then since $a$ is odd, we must have $a=1$. On the other hand, if $a+1 = 4$, then $\tau(p) = 2^k$ for some $1 \le k \le 6$, and by the recursive relation in Theorem~\ref{thm: properties} (ii), we then get $\alpha = \tau(p^3) =2^{k+1}(2^{2k-1} - p^{11})$.
However, for any odd prime $p$, we have $\abs{2^{2k-1} - p^{11}} > 175000 > \abs{\alpha}$, so this is impossible. We conclude that $n = p$ for some odd prime $p$. 

Finally, the congruences for $\tau(n)$ given in \cite[p. 4]{Mods} imply that for any prime $p \neq 23$,
\begin{align}\label{eq: pmod}
\begin{split}
    \tau(p) \equiv 0 &\pmod 2, \quad \tau(p) \equiv 0,2 \pmod 3, \quad \tau(p) \equiv 0,1,2 \pmod 5, \\
    \tau(p) &\equiv 0,1,2,4 \pmod 7, \quad \tau(p) \equiv 0,-1,2 \pmod{23}.
\end{split}
\end{align}
However, we check that each $\alpha$ in \eqref{eq: eventaulist} fails to satisfy one of these congruences and $\abs{\tau(23)} > 100$, so we cannot have $\tau(p) = \alpha$.
\end{proof}

\section{Proof of Theorem~\ref{thm: main}}\label{sec: proofofmain}
   Theorems~\ref{thm: mainodd} and \ref{thm: maineven} prove Theorem~\ref{thm: main} (v). Moreover, the argument used in the proof of Theorem~\ref{thm: maineven} shows that if $\tau(n) \in \{2,-24,-70,-90,92\}$, then $n = p$ is prime, which proves (i). To show (ii), we proceed by induction on $k$. From (i), we know that if $\tau(n) = 2$, then $n = p$ is prime. Now, if $\tau(n) = 2^k$ and $n = p_1^{a_1} \cdots p_j^{a_j}$, then since $\tau(m) \notin \{-2^k: 0 \le k \le 6\} \cup \{1\}$ for all $m > 1$ by (v), we must have $\tau(p_i^{a_i}) = 2^{c_i}$ for some $c_i > 0$ such that $c_1 + \cdots + c_j = k$. If $j > 1$, then by the induction hypothesis, we have  $a_i = 1$ and $\tau(p_i) = 2$ for all $i$, so that in particular $j = k$. Hence, it suffices to consider the case $n=p^a$. As in the proof of Theorem~\ref{thm: maineven}, we can assume $p \nmid \tau(p)$, so that $\{1, \tau(p), \tau(p^2),\cdots\}$ forms a Lucas sequence. By Lemma~\ref{lemma: defectivity}, $\tau(p^a) = 2^k$ is not a defective term. Thus since $2\mid \tau(p)$, we must have $a = 1$. However, $\tau(p) \ne 2^k$ by the congruences in \eqref{eq: pmod}, hence we cannot have $n = p^a$.
   
   The proofs of (iii) and (iv) are quite similar: Suppose first that $\tau(n) = -48$. Then given the values ruled out for $\tau(n)$ in (v), we find that either $n=p^a$ or $n = p_1^{a_1} p_2^{a_2}$ for primes $p_1,p_2$ such that $\tau(p_1^{a_1}) = 2$ and $\tau(p_2^{a_2}) = -24$ (in which case $a_1,a_2 = 1$ by part (i)). Therefore, it suffices to exclude the possibility that $\tau(p^a) = -48$. By the argument used in the proof of Theorem~\ref{thm: maineven}, we need only to consider the case $a=1$, at which point $\tau(p) = -48$ is again ruled out by the congruences in \eqref{eq: pmod}. Finally, if $\tau(n) = -96$, then given the values ruled out for $\tau(n)$ in (v) and the cases of $2,4,-24,-48$ handled above, we find that either $n=p^a$ or $n=p_1p_2p_3$ with $\tau(p_1)=\tau(p_2)=2$ and $\tau(p_3)=-24$. As before, to rule out the case $\tau(p^a) = -96$, we note that the proof of Theorem~\ref{thm: maineven} implies that we must have $a=1$, which is impossible by the congruences in \eqref{eq: pmod}. \qed

\end{document}